\documentclass[12pt]{amsart}

\usepackage{amssymb,epstopdf,subcaption,mathtools,tikz-cd,pb-diagram,thmtools, thm-restate,pinlabel,pifont,enumitem,stmaryrd}
\usepackage{graphicx}

\usepackage[all]{xy}
\usepackage[mathscr]{euscript}

\newcommand{\nc}{\newcommand}
\nc{\dmo}{\DeclareMathOperator}
\nc{\nt}{\newtheorem}

\newtheorem{theorem}{Theorem}[section]

\newtheorem{proposition}[theorem]{Proposition}
\newtheorem{corollary}[theorem]{Corollary}
\theoremstyle{definition}

\newtheorem{claim}[theorem]{Claim}

\theoremstyle{remark}

\nc{\cut}{\!\ssearrow\!}

\dmo{\Diff}{Diff}
\dmo{\Mod}{Mod}
\dmo{\SMod}{SMod}
\dmo{\I}{\mathcal{I}}
\dmo{\SO}{SO}
\dmo{\Orth}{O}
\dmo{\Sp}{Sp}
\dmo{\SL}{SL}
\dmo{\GL}{GL}
\dmo{\im}{im}
\dmo{\Emb}{Emb}
\dmo{\PSp}{PSp}
\dmo{\PSL}{PSL}
\dmo{\PMod}{PMod}
\dmo{\Homeo}{Homeo}
\dmo{\trans}{Trans}
\dmo{\Twist}{Twist}
\dmo{\Aut}{Aut}
\dmo{\Nil}{Nil}
\dmo{\Sol}{Sol}
\dmo{\Isom}{Isom}
\dmo{\Out}{Out}
\dmo{\OUT}{\mathcal{OUT}}
\dmo{\AUT}{\mathcal{AUT}}
\dmo{\Inn}{Inn}
\dmo{\ab}{ab}
\dmo{\orb}{orb}
\dmo{\Hom}{Hom}
\dmo{\Push}{Push}

\nc{\Z}{\mathbb Z}
\nc{\N}{\mathcal N}
\nc{\R}{\mathbb R}
\nc{\F}{\mathcal F}
\nc{\C}{\mathbb{C}}

\nc{\ga}{\gamma}
\nc{\de}{\delta}
\nc{\ep}{\epsilon}

\nc{\flm}{\lambda_{2}}

\nc{\normalclosure}[1]{\ensuremath{\left \langle \left \langle #1 \right \rangle \right \rangle}}

\nc{\margin}[1]{\marginpar{\scriptsize #1}}

\nc{\p}[1]{\bigskip\noindent\textbf{#1.}}
\nc{\lei}[1]{{\color{red} \sf  L: [#1]}}
\nc{\bena}[1]{{\color{blue} \sf  B: [#1]}}
\title{Circle action of the punctured mapping class group and cross homomorphism}
\author{Lei Chen}

\address{Lei Chen:  Morningside Center of Mathematics, Chinese Academy of Sciences,  \newline Academy of Mathematics and Systems Science, Chinese Academy of Sciences,  Beijing, 100190, China.}
\email{chenlei@amss.ac.cn}

\date{\today}

\begin{document}

\maketitle

\vspace*{-4ex}

\begin{abstract}
In the following short note, we give a new geometric interpretation of the generator of the infinite cyclic group $H^1(\text{Mod}(S_{g,1});H^1(S_g;\mathbb{Z}))$ (this computation is proved by Morita). There are several constructions of this class given by Earle, Morita, Trapp and Furuta. The construction we give here uses the action of $\text{Mod}(S_{g,1})$ on the circle and its rotation numbers. We also show that our construction is the same as the construction by Furuta and Trapp using winding numbers.
\end{abstract}

\section{The construction and result}
Let $S_g$ be a genus $g$ surface and let $p$ be a point on $S_g$. When $g>1$, the universal cover of $S_g$ is the hyperbolic plane $\mathbb{H}^2$. Pick a lift $\tilde{p}$ of $p$ in $\mathbb{H}^2$, any diffeomorphism $\phi$ of $S_g$ fixing $p$ can be lifted to a diffeomorphism $\tilde{\phi}$ of $\mathbb{H}^2$ fixing $\tilde{p}$. Furthermore, the lift $\tilde{\phi}$ extends to the boundary $\partial\mathbb{H}^2 \cong S^1$. The boundary action $\partial(\tilde{\phi})$ only depends on the isotopy type of $\phi$. The above construction describes the following Gromov boundary action
\[
G:\Mod(S_{g,1})\to \Homeo^+(S^1).
\]
For a proof of the above description, see \cite[Chapter 5.5]{FM}. Note that Mann--Wolff \cite{MannW} proved that any action of $\Mod(S_{g,1})$ on the circle is either trivial or semi-conjugate to $G$.

Let $\widetilde{\Homeo^+(S^1)}$  be the orientation-preserving homeomorphism of $\mathbb{R}$ that commutes with the translation by one map on $\R$. Then we have the following short exact sequence
\begin{equation}\label{Zextention}
1\to \Z\to \widetilde{\Homeo^+(S^1)}\xrightarrow{L}\Homeo^+(S^1)\to 1.
\end{equation}
Let $\trans:\widetilde{\Homeo^+(S^1)}\to \mathbb{R}$ be the translation number, which is defined as the following limit
\[
\trans(f)=\lim_{n\to \infty}\frac{f^n(0)}{n}.
\]

Let $P:\pi_1(S_g)\to \Mod(S_{g,1})$ be the point-pushing homomorphism.  Let $e: F_{2g}\to \pi_1(S_g)$ be the natural homomorphism where $F_{2g}$ is generated by $a_1,b_1,...,a_g,b_g$ and the only relation of $\pi_1(S_g)$ is given by $c:= [a_1,b_1]...[a_g,b_g]=1$. Even though the homomorphism $G \circ P$ does not have a lift under the map $L$ to $\widetilde{\Homeo^+(S^1)}$ in the exact sequence \eqref{Zextention}, the free group $G\circ P\circ e$ can be lifted to $\widetilde{\Homeo^+(S^1)}$. Let 
\[
\tilde{E}:F_{2g}\to \widetilde{\Homeo^+(S^1)}
\]
be a lift of $G\circ P\circ e$ (which is not unique).

For a group $H$, an outer-automorphism of $H$ is an automorphism of $H$ up to conjugation by an element in $H$. By Dehn--Nielsen--Baer Theorem of $\Mod(S_{g,1})$ (see, e.g., \cite[Chapter 8]{FM}), we know that $\Mod(S_{g,1})\cong \Out^*(F_{2g})$, where $\Out^*(F_{2g})$ denotes the outer-automorphism group of $F_{2g}$ consisting of elements $f:F_{2g}\to F_{2g}$ such that $f(c)$ is conjugate to $c$. By composing with a conjugation of $F_{2g}$, any element in $\Out^*(F_{2g})$ has a representative $f:F_{2g}\to F_{2g}$ such that $f(c)=c$ (such $f$ is still not unique but differed by an element in the centralizer of $c$, which is a power of $c$).

We now define a map
\[
R:\Mod(S_{g,1})\to \text{Map}(F_{2g},\mathbb{Z}).
\]
Let $\phi\in \Mod(S_{g,1})$ and $f:F_{2g}\to F_{2g}$ be a representative of $\phi$ such that $f(c)=c$. Then the definition of $R$ is given by the following formula 
\[
R(\phi)(\gamma)=\trans\circ \tilde{E}(f(\gamma))-\trans\circ \tilde{E}(\gamma)
\]
For a group $H$ and an $H$-module $M$, we call a map $\rho:H\to M$ a cross homomorphism if
\[
\rho(\gamma_1\gamma_2)=\rho(\gamma_1)+ \gamma_1\rho(\gamma_2).
\]
The main result of this paper is the following.
\begin{theorem}\label{main}
The map $R$ is well-defined, has image in $\text{Hom}(F_{2g},\mathbb{Z})\cong H^1(\pi_1(S_g);\Z)$ and $R$ is a cross homomorphism. Furthermore, as an element in the cohomology class, $[R]$ is a generator of $H^1(\Mod(S_{g,1}); H_1(\pi_1(S_g);\Z))$, which is equal to the cross homomorphism defined by Morita \cite{Morita}.
\end{theorem}
Notice that Earle \cite{Earle}, Morita \cite{Morita}, Trapp \cite{Trapp} and Furuta \cite{Furuta} (described by Morita) gave various constructions of the generator of the group \[
H^1(\Mod(S_{g,1}); H_1(\pi_1(S_g);\Z))\cong \Z.\] 
Trapp and Furuta used the unit tangent bundle and winding numbers. Earle used Moduli spaces and Morita gave a combinatorial construction. All of the above definition seem very different. It is an interesting question to consider how those construction interact with each other. Kuno \cite{Kuno} gave an understanding of the difference between Earle construction and Morita construction. This cross homomorphism is also related to Johnson homomorphism as discussed in Morita \cite{Morita2}.

Our second result is to relate our construction with that of  Trapp \cite{Trapp} and Furuta \cite{Furuta} (described by Morita). Let $D\subset S_g$ be an open disk. Let $X$ be a nowhere zero vector field on $S_g-D$. Let $\mathcal C$ be the set of nonseparating simple closed curves on $S_g-D$. Let $\mathcal P$ be the set of all closed curves on $S_g-D$.

Then there is a winding number map 
\[
\omega_X:\mathcal C\subset  \mathcal P \to\Z 
\]
such that $\omega_X(c)$ is the number of times the tangents to $c$ rotate with respect to the framing of $c$ by $X$. For example, if $z$ is a trivial loop, then $\omega_X(z)=1$.

Our next result is the following comparison between $\omega_X$ and $\trans\circ \tilde{E}$.
\begin{theorem}\label{main2}
The map $\trans\circ \tilde{E}$ is conjugation-invariant, which gives a map
\[
\trans\circ \tilde{E}:\mathcal C\subset  \mathcal P \to\Z.
\]
Furthermore, there exists a nowhere zero vector field $X$ on $S_g-D$ such that
\[
\omega_X|_{ \mathcal C}=\trans\circ \tilde{E}|_{ \mathcal C}.
\]
\end{theorem}
We remark that over $ \mathcal L$, the map $\omega_X$ and $\trans\circ \tilde{E}$ are not the same since for a trivial loop $z$, as we have $\trans\circ \tilde{E}(z)=0$. It is different from $\omega_X(z)=1$.

\begin{corollary}\label{main3}
The map $R$ is equal to the cross homomorphism defined by Trapp \cite{Trapp}.
\end{corollary}

\p{Acknowledgement}
We thank Bena Tshishiku for introducing me the paper of Morita. We also thank Nick Salter and Dan Margalit for helpful comments on an earlier version of the paper.

\section{Background: results of Matsumoto and Humphries--Johnson}
We will introduce two old results of Matsumoto and Humphries--Johnson.
\subsection{The result of Matsumoto}
For the group $\Homeo^+(S^1)$, we have the following canonical Euler cocycle \[
\tau: \Homeo^+(S^1)\times \Homeo^+(S^1)\to \Z\] given by
\[
\tau(f,g):=\trans(\tilde{f}\tilde{g})-\trans(\tilde{f})-\trans(\tilde{g}),\]
where $\tilde{f}$ and $\tilde{g}$ are lifts of $f,g$ in $\widetilde{\Homeo}^+(S^1)$. The map $\tau$ does not depend on the lifts and the value of $\tau$ is in $\{-1,0,1\}$.

Matsumoto \cite[Theorem 3.3]{Matsumoto} proved the following.
\begin{theorem}[Matsumoto]\label{Matsumoto}
The cocycle $\tau$ is the same as a geometric cocyle \[\theta: \pi_1(S_g)\times \pi_1(S_g)\to \mathbb{Z}\]
when restricted to the Fuchsian representation $\pi_1(S_g)\to \Homeo^+(S^1)$.
The cocycle $\theta$ is given by the following rule: Suppose that the subgroup $F$ generated by $\alpha,\beta$ is a free group on two generator, consider the covering $U\to S_g$ associated with $F$, $U$ is either a punctured torus or a pair of pants. If $U$ is a pair of pants and that $\alpha$ and $\beta$ are represented by two boundary curves which match (resp. oppose) the orientation of $U$, we define $\theta(\alpha,\beta)=-1 (resp. -1)$. In all the other cases, define $\theta(\alpha,\beta)=0$. 
\end{theorem}

%We now recall the result of Matsumoto \cite[Theorem 3.3]{Matsumoto}, which computes
%\[
%D(\trans\circ \tilde E)(a,b):=\trans\circ \tilde E(ab)-\trans\circ \tilde E(a)-\trans\circ \tilde E(b)
%\]
%for $a,b\in F_{2g}$.
%\begin{theorem}[Matsumoto]
%Suppose that  the subgroup $F$ generated by $a,b$ is a free group on two generators. Consider 
%the covering $U\to S_g$ associated with $F$. The surface $U$ is then either a punctured torus or 
%a pair of pants. Suppose the latter. Consider the loop in $U$ representing $a$ and 
%$b$. If $a$ and $b$ are represented by two boundary curves which match (resp. oppose) 
%the orientation of $U$ derived from that of $S_g$ and by two joining paths which 
%are simple and disjoint, we define $D(\trans\circ \tilde E)(a,b)=-1$ (resp. $+1$). In all the other 
%cases define $D(\trans\circ \tilde E)(a,b)=0$. 
%\end{theorem}
Let $M=S_g-D$. Since $G\circ P$ comes from a Fuchsian representation, and $\tilde E$ is a lift of $G\circ P$ to $\widetilde{\Homeo}^+(S^1)$, we have the following consequence.
\begin{proposition}\label{Matsumoto1}
Let $a,b,c$ be three simple closed curves on $M$ bounding a pair of pants $Q$ in $M$ such that $Q$ is on the left of each of $a,b,c$. We have that
\[
\trans\circ \tilde E(a)+\trans\circ \tilde E(b)+\trans\circ \tilde E(c) = -1
\]
\end{proposition}
\begin{proof}
Since $c$ is equal to $ab$ with a change of orientation, we have
\[
\trans\circ \tilde E(c)+\trans\circ \tilde E(a)+\trans\circ \tilde E(b) = -1,
\]
which does not depend on the lift $\tilde E$ by Theorem \ref{Matsumoto}.
\end{proof}
We have another corollary. 
\begin{proposition}\label{Matsumoto2}
\begin{enumerate}
\item
Let $a,b\in \pi_1(S_g)$ be two element having nonzero geometric intersection number. Then we have
\[
\trans\circ \tilde E(a)+\trans\circ \tilde E(b)=\trans\circ \tilde E(ab)
\]
\item Let $a,b\in \pi_1(S_g)$ be two element that are homotopic. Then we know 
\[
\trans\circ \tilde E(a)+\trans\circ \tilde E(b)=\trans\circ \tilde E(ab)
\]
\end{enumerate}
\end{proposition}
\begin{proof}
This comes from the fact that the liftings of $a,b$ will have intersection number, which implies that $\theta(a,b) = 0$. The other case follows similarly.\end{proof}

\subsection{The result of Humphries--Johnson}
Humphris--Johnson \cite{HJ} classify the set of \emph{twist linear functions} which is defined as the following. Let $\mathcal C$ be the set of nonseparating simple closed curves. For two curves $c,d$ on a surface $M$, let $\langle d,c\rangle$ denote the algebraic intersection number. 

Then we can define a twist linear function, with values in an abelian group $V$, to be a function
\[
\phi: \mathcal C\to V
\]
satisfying
\[
\phi(T_d(c))=\phi(c)+\langle d,c\rangle\phi(d)
\]
Let $UM$ be the unit tangent bundle of $M$. For any class $\alpha\in H^1(UM;\Z)$, we can define
\[
\omega_\alpha: \mathcal C\to \Z
\]
by lifting a smooth representative and the evaluate on $\alpha$. Humphries--Johnson \cite[Theorem 2.1]{HJ} showed that any twist linear map is given by $\omega_\alpha$. Furthermore, if $\omega_\alpha(z)=1$ for the trivial loop $z$, we have that $\omega_\alpha=\omega_X$ for a nowhere zero vector field on $M$.

\section{The proof of Theorem \ref{main}}
\p{\boldmath Step 1: $R$ is well-defined, and has image in $\text{Hom}(F_{2g},\mathbb{Z})\cong H^1(\pi_1(S_g);\Z)$}

Let $\phi\in \Mod(S_{g,1})$. To prove that $R(\phi)$ is well-defined, we need to prove that it does not depend on the choice of $f:F_{2g}\to F_{2g}$ representing $\phi\in \Out^*(F_{2g})$ such that $f(c)=c$. This comes from the fact that a different choice is given by composing $f$ with a conjugation by a power of $c$ where $\tilde{E}(c)$ is a translation by integers. Conjugation by a translation of integers has no effect on the translation number since translation by integers commute with $\widetilde{\Homeo^+(S^1)}$.

We now show that $R$ has image in $\text{Hom}(F_{2g},\mathbb{Z})\cong H^1(\pi_1(S_g);\Z)$.
\begin{claim}Let $\phi\in \Mod(S_{g,1})$ and $\alpha,\beta\in F_{2g}$, then
\[
R(\phi)(\alpha \beta)=R(\phi)(\alpha) + R(\phi)(\beta)\]
\end{claim}
\begin{proof}
Let $f:F_{2g}\to F_{2g}$ be a representative of $\phi\in \Out^*(F_{2g})$ such that $f(c)=c$.

The formula $\theta(\alpha,\beta)$ is equivariant under the action of the mapping class group $\Mod(S_{g,1})$. Thus from the identification of $\theta$ and $\tau$ by Theorem \ref{Matsumoto}, we obtain that \[
\tau(\phi(\alpha),\phi(\beta))=\tau(\alpha,\beta).\]
This implies that
\begin{equation}
\begin{split}
&\trans(\tilde{E}(\alpha)\tilde{E}(\beta))-\trans(\tilde{E}(\alpha))-\trans(\tilde{E}(\beta)) \\
 =&  \trans(\tilde{E}(f(\alpha))\tilde{E}(f(\beta)))-\trans(\tilde{E}(f(\alpha)))-\trans(\tilde{E}(f(\beta))) 
\end{split}
\end{equation}

\begin{equation}
\begin{split}
&\trans(\tilde{E}(f(\alpha))\tilde{E}(f(\beta))) - \trans(\tilde{E}(\alpha)\tilde{E}(\beta)) \\
 =& \trans(\tilde{E}(f(\alpha))) -\trans(\tilde{E}(\alpha))+ \trans(\tilde{E}(f(\beta)))  -  \trans(\tilde{E}(\beta))
\end{split}
\end{equation}
Then 
\[
R(\phi)(\alpha \beta)=R(\phi)(\alpha)+R(\phi)(\beta). \qedhere\]
\end{proof}

\p{\boldmath Step 2: $R$ is a cross homomorphism}
We need to show that 
\[
R(\phi\eta) = \eta^{-1}R(\phi) + R(\eta)
\]
Let $f,h:F_{2g}\to F_{2g}$ be representative of $\phi,\eta$ such that $f(c)=h(c)=c$. Then $h\circ f$ is a representative of $\phi\eta$ such that $h\circ f(c)=c$. Then
\begin{equation}
\begin{split}
 &R(\phi\eta)(\alpha) \\
=& \trans(h\circ f(\alpha))-\trans(\alpha)\\
= & \trans(h(f(\alpha))) - \trans(f(\alpha)) + \trans(f(\alpha))-\trans(\alpha)\\
=& R(\eta) (f(\alpha)) + R(\phi)(\alpha)\\
= & \phi^*R(\eta)(\alpha) + R(\phi)(\alpha)
\end{split}
\end{equation}
This proves what we need.

\p{\boldmath Step 3: $R$ is the generator of $\Mod(S_{g,1})\to H^1(S_g;\Z)$}
For this part, we only need to check the restriction on the point-pushing subgroup $P:\pi_1(S_g)\to \Mod(S_{g,1})$. This gives a homomorphism
\[
R|_{P}:\pi_1(S_g)\to H^1(S_g;\Z),
\]
since the point-pushing subgroup acts trivially on $H^1(S_g;\Z)$. The homomorphism $R|_{P}$ factors through the abelianization $H_1(S_g;\mathbb{Z})$ of $\pi_1(S_g)$. Let $i:\pi_1(S_g)\times \pi_1(S_g)\to \Z$ be the algebraic intersection number, we have the following.
\begin{claim}We have the following formula
\[
R|_P(a)(b) = (2-2g)i(a,b).\]
\end{claim} 
\begin{proof}
We will do this by a calculation. Let $a_1,b_1,...,a_g,b_g$ be the standard generating set of $\pi_1(S_g)$. The point-pushing $P(a_1)$ has the following representative $f:F_{2g}\to F_{2g}$ such that 
\begin{equation}
a_1\to a_1, b_1\to a_1^{-1}cb_1a_1, a_j \to a_1^{-1}ca_jc^{-1}a_1, b_j \to a_1^{-1} cb_jc^{-1}a_1.
\end{equation}
We obtain the above representative by first find some representative using the exact mapping class, then we compose with a conjugation to make sure that $f(c)=c$.

Now we can easily obtain that $R(P(a_1))(a_1) = 0$, $R(P(a_1))(b_1) = 2-2g$ and $R(P(a_1))(a_j) = 0, R(P(a_1))(b_j) = 0$ for $1<j\le g$. Thus we obtain that \[R(P(a_1))(b)=(2-2g)i(a_1,b).\] 
Let $c$ be an element in $H_1(S_g;\Z)$ represented as a non-separating simple closed curve. Then there is an element $\phi\in \Mod(S_{g,1})$ such that $\phi(a_1)=c$. Then we have that
\[
R(\phi^{-1} P(a_1) \phi) = R(P(c))
\]
By cross homomorphism of $R$, we know that
\[
R(\phi^{-1} P(a_1) \phi) = R(\phi^{-1}P(a_1) ) + (\phi^{-1} P(a_1) )^*R(\phi) = R(\phi^{-1}) + (\phi^*)^{-1}(P(a_1))+(\phi^*)^{-1}(R(\phi))
\]
Since $R(\phi^{-1} \phi) = R(id) = 0$, we know that $R(\phi^{-1}) +(\phi^*)^{-1}(R(\phi^{-1})) = 0$.

Thus we obtain that $R(P(c))=(\phi^*)^{-1}(R(P(a_1)))$, then we have
\[
\begin{split}
&R(P(c))(x) = (\phi^*)^{-1}(R(P(a_1)))(x) = R(P(a_1))((\phi_*)^{-1}(x))\\
=& i(a_1,(\phi_*)^{-1}(x))=i(\phi_*(a_1),x) = i(c,x)\qedhere
\end{split}
\]\end{proof}
Then by Morita \cite[Section 4]{Morita}, we know that $R$ as a cohomology class in $H^1(\Mod(S_{g,1}),H^1(S_g;\Z))$ is the same as Morita's class.

\section{The proof of Theorem \ref{main2} and \ref{main3}}
The following is well-known.
\begin{proposition}
\[\trans: \widetilde{\Homeo^+(S^1)} \to \R\]
is conjugation-invariant.
\end{proposition}
\begin{proof}
This is because the translation number is the long-term behavior of an element under iterations, which implies that it is conjugation-invariant.
\end{proof}
We denote by $\tau:= \trans \circ \tilde E$. We thus know that 
\[\tau: F_{2g} \to \Z \]
only depends on the conjugation classes of $F_{2g}$, or on closed curves on $M$. Let $X$ is a nowhere zero vector field on $M$. Both $\tau$ and $\omega_X$ are functions from $\mathcal L$ to $\Z$. However the difference between $\omega_X$ and $\tau$ is that for a trivial loop $z$, we have $\omega_X(z)=1$ but $\tau(z)=0$.

We now prove that $\tau$ is a twist linear function. However, the map $\tau$ is not cross linear as defined by Humphries--Johnson \cite{HJ}. 

\begin{proposition}\label{twistlinear}
$\tau$ is a twist linear function.
\end{proposition}
\begin{proof}
Let $c$ be a curve starting from the basepoint on $\partial M$. Let $d$ be a simple closed curve on $M$ such that $i(c,d)\neq 0$. Then we know that  $i(T_d(c),c)=i(d,c)^2\neq 0$ by Farb--Margalit \cite[Proposition 3.2]{FM}. In this case, we know that $T_d(c)$ and $c$ has nonzero geometric intersection number on $S_g$. We then know
\[\tau(T_d(c)c^{-1}) = \tau(T_d(c)) + \tau(c^{-1}) \]
by Proposition \ref{Matsumoto2}.

Let $c$ be the union $c_1...c_k$ where each $c_i$ is a component of $c-d$. Then we have
\[
T_d(c) = c_1d^{e_1} c_2d^{e_2}...c_k
\]
for $e_1,...,e_{k-1}\in \pm 1$ depending on the orientation of the intersection $c\cap d$ and $\sum e_i = \langle d,c\rangle$. We thus know 
\[
T_d(c) = c_1d^{e_1} c_2...c_k = c_1d^{e_1}c_1^{-1} c_1c_2 d^{e_2} c_2^{-1}c_1^{-1} ...c_1...c_k
\]
We know that the element $d_k:=c_1...c_k d c_k^{-1}...c_1^{-1}$ is homotopic to $d$.

Then we know by Proposition \ref{Matsumoto2} that
\[
\tau(T_d(c)c^{-1}) = \langle d,c \rangle \tau(d)
\]
Combining the above two equations, we know 
\[ \tau(T_d(c)) = \tau(T_d(c)c^{-1}) -  \tau(c^{-1})=  \tau(c) + \langle d,c \rangle \tau(d) \]
This proves that $\tau$ is a twist linear function.
\end{proof}

We now start the proof of Theorem \ref{main2}.
\begin{proof}[Proof of Theorem \ref{main2}]
Humphries--Johnson \cite[Theorem 2.1]{HJ} showed that any twist linear map over $\mathcal C$ is given by $\omega_\alpha$.  Since $\tau$ is a twist linear function, there exists a class $\alpha\in H^1(US_g^1;\Z)$ such that $\omega_\alpha = \tau$. We now compute $\omega_\alpha(z)$ for a trivial loop $z$.

The trivial loop $z$ corresponds to a pair of pants $a,b,c$ such that $-z=a+b+c$ by Humphries--Johnson \cite[Lemma 2.4]{HJ}. Then by Proposition \ref{Matsumoto1}, we have
\[\omega_\alpha(z)=-(\tau(a)+\tau(b)+\tau(c)) = 1.\]
Thus by  Humphries--Johnson \cite[Theorem 2.1]{HJ}, we have that $\omega_\alpha=\omega_X$ for a vector field $X$ over $S_g^1$ (vector fields on $S_g^1$ corresponds to the twist linear function whose value is $1$ on the trivial loop).
\end{proof}

\section{Another interpretation of the crossed homomorphism $R$}
In this section, we will use Section 4 to give a more precise statement about which vector field should be the one in Theorem \ref{main2}.
\subsection{A definition of $R$ without making choices}
While the action $G$ does not lift to $\widetilde{\Homeo^+(S^1)}$, we have a unique lifting for the $\Z$-central extension $\Mod(S_g^1)$ as the following by \cite[Section 5.5.6]{FM}
\[
\tilde{G}: \Mod(S_g^1)\to \widetilde{\Homeo^+(S^1)}
\]
The subgroup of $\Mod(S_g^1)$ which is the kernel of the forgetful map $\Mod(S_g^1)\to \Mod(S_g)$ is the disk-pushing subgroup which is isomorphic to $\pi_1(US_g)$ where $US_g$ is the unit-tangent bundle of $S_g$. By composing with the $\trans$, we define the following map
\[
\mathcal T: \Mod(S_g^1)\to \widetilde{\Homeo^+(S^1)}\xrightarrow{\trans}\Z
\]
Let $US_g$ be the unit tangent bundle of the surface $S_g$. We know that $\Mod(S_g^1)$ contains $\pi_1(US_g)$ as a normal subgroup such that the quotient is $\Mod(S_g)$. Since the center element acts trivially on $\pi_1(US_g)$, we obtain an action of $\Mod(S_{g,1})$ on $\pi_1(US_g)$. More precisely, we have a result from  \cite{BenaChen} as the following.
\begin{theorem}
The two groups $Aut(\pi_1(US_g))$ and $Aut(\pi_1(S_g))$ satisfying the following short exact sequence 

\[
1\to H^1(S_g;\Z)\to \Aut(\pi_1(US_g))\to \Aut(\pi_1(S_g))\]
and it has a natural splitting $A:\Aut(\pi_1(S_g))\to \Aut(\pi_1(US_g))$. The map $A(\phi)$ is defined as the isomorphism on $\pi_1(US_g)$ of the representative $\phi:F_{2g}\to F_{2g}$ such that $\phi(c)=c$. 
\end{theorem}
Using the action $A$, we give another definition of $R$. Let $\phi\in \Mod(S_{g,1})$ and $\alpha\in \pi_1(US_g)$, we define
\[
R(\phi) (\alpha) =\mathcal T(A(\phi)\alpha) - \mathcal T(\alpha).
\]
By definition, $R$ is the same as the one defined in the introduction. 
In here we obtain a map
\[
\mathcal T|_{\pi_1(US_g)}: \pi_1(US_g)\to \Z.
\]
So far, we have not made any choices. In here our definition of $\Mod(S_{g,1})\to H$ has nothing to do with the choice of lifting $\tilde E$.

\subsection{\boldmath The lifting $\tilde E$ given by an embedding $i: S_g^1\to S_g$.}

Let $i: S_g^1\to S_g$ be an embedding. We then obtain another embedding $\tilde i:US_g^1\to US_g$. Let $X: S_g^1\to US_g^1$ be a vector field. Then the vector field induced a homomorphism
\[
X_*: \pi_1(S_g^1)\to \pi(US_g^1).
\]
Now the map \[
\tau := \mathcal T\circ \tilde i\circ X_*: \pi_1(S_g^1)\to \Z
\]
is the composition of a lifting $\tilde G\circ \tilde i\circ X_*: \pi_1(S_g^1) \cong F_{2g}\to \widetilde{\Homeo^+(S^1)}$ (a homomorphism) and $\trans$. 

For a curve $\alpha\in \pi_1(S_g^1)$ (not necessarily simple), we have another lift $L(\alpha)\in \pi_1(US_g^1)$ by lifting the tangents. However this is no longer a homomorphism since the trivial loop will lift to the center. We now prove the following.
\begin{proposition}
For any element $\alpha\in\pi_1(S_g^1)$ (not necessarily simple), we have
\[
\tau(\alpha) = \mathcal T \circ \tilde i(L(\alpha)) - \omega_X(\alpha)
\]
\end{proposition}
\begin{proof}
Let $z$ be the center of $\pi_1(US_g^1)$. By the definition of winding number, we have
\[
L(\alpha)- X_*(\alpha)=\omega_X(\alpha)z.\]
Then we have

\[
\tau(\alpha)=\mathcal T \circ \tilde i(X_*(\alpha)) = \mathcal T \circ \tilde i(L(\alpha)-\omega_X(\alpha)z) = \mathcal T \circ \tilde i(L(\alpha)) - \omega_X(\alpha)
\]
For the trivial loop $\epsilon$, we have that both $\mathcal T \circ \tilde i(L(\alpha))$ and $\omega_X(\alpha)$ is $1$, and the difference happen to be the map $\tau$. 
\end{proof}

For $\alpha$ a simple closed curve in $\pi_1(S_g^1)$, the way to find $L(\alpha)$ is through taking a geodesic representative and the lifting it to $US_g$.  

\begin{proposition}
For $\alpha\in \pi_1(S_g^1)$ a simple closed curve, we have 
\[
\mathcal T \circ \tilde i\circ L(\alpha) = 0
\]
and 
\[
\tau(\alpha) = -\omega_X(\alpha).
\]
\end{proposition}
\begin{proof}
Since both $\tau$ and $\omega_X$ are twist-linear map, we know that their difference is also a twist linear map. Then  $\mathcal T \circ \tilde i\circ L$ is a twist linear map. 

Since the action of $\Mod(S_g^1)$ on the simple closed curve in $S_g^1$ is transitive, and the translation number is a conjugation invariant, we know that $\mathcal T \circ \tilde i\circ L$ on non-separating simple closed curves are all the same. By the Dehn twist formula of twist linear functions, we know that $\mathcal T \circ \tilde i\circ L$ is zero on non-separating simple closed curves.
\end{proof}

\bibliographystyle{alpha}
\bibliography{citing}

\end{document}